\documentclass[12pt]{article}
\usepackage{amssymb}
\usepackage{enumerate}
\usepackage{amsmath}
\usepackage{amsthm}

\newtheorem{theorem}{Theorem}[section]
\newtheorem{proposition}{Proposition}[section]

\newtheorem{lemma}[theorem]{Lemma}
\newtheorem{remark}[theorem]{Remark}

\title{Local well-posedness and wave breaking results for periodic solutions of a shallow water equation for waves of moderate amplitude}
\author{Nilay Duruk Mutluba\c{s}\\
University of Vienna, Faculty of Mathematics,\\
Nordbergstrasse 15, 1090 Vienna, Austria\\
nilay.duruk.mutlubas@univie.ac.at}
\begin{document}

\maketitle
 \begin{abstract}
  \noindent 
  We study the local well-posedness of a periodic nonlinear equation for surface waves of moderate amplitude in shallow water. We use an approach due to Kato which is based on semigroup theory for quasi-linear equations. We also show that singularities for the model equation can occur only in the form of wave breaking, in particular surging breakers.
 \end{abstract}
 
 
 \section{Introduction}
We are concerned with an evolution equation which models the propagation of surface waves of moderate amplitude in the shallow water regime:
   \begin{align}
    \label{MAE1}
  & u_t + u_x + \frac{3}{2}\varepsilon u u_x -\frac{3}{8}\varepsilon^{2} u^2u_x + \frac{3}{16}\varepsilon^{3} u^3 u_x+\frac{\mu}{12}(u_{xxx}-u_{xxt})\nonumber \\
  	&+\frac{7\varepsilon}{24}\mu (uu_{xxx} +2 u_xu_{xx})=0,~~~~x\in\mathbb{R},~t>0.
  \end{align}
Here $u(x,t)$ is the free surface elevation and $\varepsilon$ and $\mu$ represent the amplitude and shallowness parameters, respectively.   

Since the exact governing equations for water waves fail to provide explicit solutions, many approximate model equations have been proposed, which are based on linear theory and were therefore inadequate to explain potential nonlinear behaviours like wave breaking or solitary waves. Hence, many competing nonlinear models have been suggested to manage these phenomena. One of the most prominent examples is the Camassa-Holm (CH) equation \cite{CH93} which is an integrable, infinite-dimensional, Hamiltonian system \cite{Con01, CGI, MKST}. A significant feature of the CH equation is that some of the bounded classical solutions develop singularities in finite time in the form of wave breaking, i.e. their slope becomes unbounded \cite{CE98}. Beyond the breaking time, the solutions recover in the sense of global weak solutions \cite{BC07,BC07II}. The relevance of the CH equation as a model for the propagation of shallow water waves was discussed by Johnson \cite{J02}, where he showed that the CH equation describes 
the horizontal component of the velocity field at a certain depth within the fluid; see also \cite{Con11}. Following the ideas presented therein, Constantin and Lannes derived the evolution equation (\ref{MAE1}) for the free surface which approximates the governing equations to the same order as the CH equation \cite{ConLn09}.

Local well-posedness results for the initial value problem associated to (\ref{MAE1}) was first proved by Constantin and Lannes \cite{ConLn09} for initial data $u_0\in H^{s+1}$, $s>3/2$. It has been recently shown that local well-posedness can also be obtained for a class of initial data comprising less regular data $u_0\in H^s$, $s>3/2$ additionally \cite{NDM13}. Enlarging the class enhances finding initial profiles that develop singularities in finite time. The model equation also possesses solitary travelling wave solutions decaying at infinity \cite{G12}. Their orbital stability has been recently studied using an approach proposed by Grillakis, Shatah and Strauss \cite{DG13}, and taking advantage of the Hamiltonian structure of (\ref{MAE1}).

In the present paper, we look for solutions of the Cauchy problem corresponding to (\ref{MAE1}) which are spatially periodic of period 1. We employ an approach due to Kato using semigroup theory for quasi-linear equations. Furthermore, we prove that singularities arise in finite time in the form of breaking waves.

Before we give the local well-posedness results, we state the theory Kato proposed:

\section{Kato's theory}
\label{Katotheory}
Consider the abstract quasi-linear evolution equation in the Hilbert space $X$:
  \begin{equation}
  \label{qlee}
     u_t =A(u)u + f(u), ~~~~t\geq 0,~~~~~u(0)=u_0.
  \end{equation}
Let $Y$ be a second Hilbert space such that $Y$ is continuously and densely injected into $X$ and let $S:Y\rightarrow X$ be a topological isomorphism. Assume that

\begin{itemize}
\item[(A1)] Given $C>0$, for every $y\in Y$ with $||y||_Y\leq C$, $A(y)$ is quasi-m-accretive on $X$, i.e.
$A(y)$ is the generator of a $C_0$ semigroup $\{T(t)\}_{t\geq 0}$ in $X$ satisfying $||T(t)||\leq Me^{\omega t}$ with $M=1$.
\item[(A2)] For every $y\in Y$, $A(y)$ is a bounded linear operator from $Y$ to $X$ and
\begin{displaymath}
||(A(y)-A(z))\omega||_X\leq c_1||y-z||_X||\omega||_Y,~~~~y,z,\omega\in Y.
\end{displaymath}
\item[(A3)] For every $C>0$, there is a constant $c_2(C)$ such that $SA(y)S^{-1}=A(y)+B(y)$, for some bounded
linear operator $B(y)$ on $X$ satisfying
 \begin{displaymath}
||(B(y)-B(z))\omega||_X\leq c_2(C)||\omega||_X,~~~~\omega\in X.
\end{displaymath}
\item[(A4)] The function $f$ is bounded in $Y$ and Lipschitz in $X$ and $Y$, i.e.
\begin{equation*}
||f(y)||_Y\leq M
\end{equation*}
for some constant $M>0$, as well as
\begin{equation*}
||f(y)-f(z)||_X\leq c_3||y-z||_X,~~~~~\forall y,z\in X
\end{equation*}
and
\begin{equation}
||f(y)-f(z)||_Y\leq c_4||y-z||_Y,~~~~~\forall y,z\in Y.\label{lip}
\end{equation}
\end{itemize}
Here $c_1$, $c_2$, $c_3$ and $c_4$ are non-negative constants.

\begin{theorem} \cite{KatoI, KatoII} 
\label{kato} 
Assume (A1), (A2), (A3), (A4) hold.
Given $u_0\in Y$, there is a maximal $T>0$, depending on $u_0$, and a unique solution $u$ to (\ref{qlee}) such that
\begin{displaymath}
u=(u_0,.)\in C([0,T),Y)\cap C^1([0,T),X).
\end{displaymath}
Moreover, the map $u_0\mapsto u(u_0,.) $ is continuous from $Y$ to \\
\mbox{$C([0,T),Y)\cap C^1([0,T),X)$}.
\end{theorem}

 \section{Local Well-Posedness}
Consider the following periodic Cauchy problem where (\ref{MAE1}) is rewritten in quasi-linear equation form:
  \begin{align}
   \label{SG}
    & u_t =(\partial_x+\frac{7}{2}\varepsilon u\partial_x )u+f(u)~~~~~x\in\mathbb{R},~t>0, \\
    & u(x,0)= u_0(x) ~~~~~x\in\mathbb{R},\label{IV}\\
    & u(x,t)=u(x+1,t)~~~~~x\in\mathbb{R},~t>0.\label{Per}
  \end{align}
 Here,
  \begin{align}
f(u)= -(1-\frac{\mu}{12}\partial_x^2)^{-1}\partial_x[2u+\frac{5}{2}\varepsilon u^{2}-\frac{1}{8}\varepsilon^{2}u^{3}+\frac{3}{64}\varepsilon^{3}u^{4}-\frac{7}{48}\varepsilon\mu u_x^2]. \label{f}
  \end{align}
Hereafter, we denote the Sobolev space of functions of period 1 by $H_p^s$.
 \begin{proposition}
 \label{Lwp}
   Let $u_0\in H_p^s$, $s>\frac{3}{2}$ be given. Then there exists $T>0$, depending on $u_0$, such that there is a unique solution $u$ to
   (\ref{SG})-(\ref{Per}) satisfying
   \begin{equation*}
    u=u(u_0,.)\in C([0,T), H_p^s)\cap C^1([0,T),L^{2}[0,1]).
   \end{equation*}
Moreover, the map $u_0\in H_p^s\mapsto u(u_0,.) $ is continuous from $H_p^s$ to \\
\mbox{$C([0,T),H_p^s)\cap C^1([0,T),L^2[0,1])$}.
  \end{proposition}

  To prove well-posedness of the solution locally in time which is proposed above, we will apply the approach mentioned in Section \ref{Katotheory}, for $X=L^{2}[0,1]$, $Y=H_p^s$ with $s>3/2$ and $S=\Lambda^{s}$ with $\Lambda=(1-\partial_x^2)^{1/2}$. First of all, we need the following lemmas ensuring the validity of the assumptions (A1)-(A4). For convenience, we neglect the constant coefficients of the terms appearing in the evolution equation.

  \begin{lemma}
The operator $A(u)=u\partial_x+\partial_x$, with domain
\begin{displaymath}
 \mathcal{D}(A)=\{\omega\in L^{2}[0,1]:(1+u)\partial_x\omega\in L^2[0,1]\}\subset L^{2}[0,1]
\end{displaymath}
is quasi-m-accretive if $u\in H_p^s$, $s>3/2$.
\end{lemma}

\begin{proof}
 Let $X$ be a Hilbert space. A linear operator $A$ in $X$ is quasi-m-accretive if \cite{KatoII}
 \begin{enumerate}
  \item [(a)] There is a real number $\beta$ such that $(A\omega,\omega)_X\geq -\beta||\omega||_X^2$ for all $\omega\in D(A)$.
  \item [(b)] The range of $A+\lambda I$ is all of $X$ for some (or equivalently, all) $\lambda>\beta$.
 \end{enumerate}
Consider the $L^2$ inner product
\begin{align*}
 (A(u)\omega,\omega)_{L^{2}[0,1]}=& (u\partial_x\omega+\partial_x\omega,\omega)_{L^{2}[0,1]}.
\end{align*}
Using integration by parts, periodicity and a Sobolev embedding theorem, we obtain
\begin{align*}
(u\partial_x\omega,\omega)_{L^{2}[0,1]}+(\partial_x\omega,\omega)_{L^{2}[0,1]}&=-\frac{1}{2}(\omega\partial_x u,\omega)_{L^{2}[0,1]}\\
&\geq -\frac{1}{2}||u_x||_{L^\infty[0,1]}(\omega,\omega)_{L^{2}[0,1]}\\
&\geq -\frac{1}{2}||u||_{H_p^s}||\omega||_{L^{2}[0,1]}^2.
\end{align*}
Choosing $\beta=\frac{1}{2}||u||_{H_p^s}$, the operator satisfies the inequality in (a). Thus, $A(u)+\lambda I$ is dissipative for all $\lambda>\beta$. Moreover, observe that $A(u)$ is a closed operator. In the view of these facts, $A(u)+\lambda I$ has closed range in $L^2[0,1]$ for all $\lambda>\beta$. Hence, in order to prove (b), it is enough to show that $A(u)+\lambda I$ has dense range in $L^2[0,1]$ for all $\lambda>\beta$. We use the fact that if adjoint of an operator has trivial kernel, then the operator has dense range. 

For $A(u)=u\partial_x+\partial_x$, the adjoint operator is given by $A^*=-\partial_x(u+1)$ with
\begin{equation*}
\mathcal{D}(A^*)=\{\omega\in L^2[0,1]:-\partial_x((u+1)\omega)\in L^2[0,1]\}.
\end{equation*}
Using the Leibniz formula, we get
\begin{equation*}
\partial_x((u+1)\omega)=\omega\partial_x u+u\partial_x\omega+\partial_x\omega.
\end{equation*}
Therefore, since $u_x\in L^{\infty}[0,1]$ and $\omega\in L^2[0,1]$,
\begin{align*}
&\mathcal{D}(A)=\{\omega\in L^{2}[0,1]:(1+u)\partial_x\omega\in L^2[0,1]\}=\\
&\mathcal{D}(A^*)=\{\omega\in L^2[0,1]:-\partial_x((u+1)\omega)\in L^2[0,1]\}.
\end{align*}
Now, assume that the range of $A(u)+\lambda I$ is not all of $L^2[0,1]$. Then there exists $0\neq z\in L^2[0,1]$
such that $((A(u)+\lambda I)\omega,z)_{L^2}=0$ for all $\omega\in \mathcal{D}(A)$. Since $H_p^1\subset \mathcal{D}(A)$, $\mathcal{D}(A)=\mathcal{D}(A^*)$ is dense in $L^2[0,1]$. This means that there exists a sequence $z_k\in \mathcal{D}(A^*)$  which converges to an element $z\in L^2[0,1]$. Recall that $\mathcal{D}(A^*)$ is closed. Therefore, $z\in \mathcal{D}(A^*)$. Moreover,
\begin{equation*}
((A(u)+\lambda I)\omega,z)_{L^2}=(\omega,(A(u)+\lambda I)^*z)_{L^2}=0
\end{equation*}
reveals that $A^*(u)+\lambda z=0$ in $L^2$. Multiplying by $z$ and integrating by parts, we get
\begin{equation*}
0=((A^*(u)+\lambda I)z,z)_{L^2}=(\lambda z,z)_{L^2}+(z,A(u)z)_{L^2}\geq (\lambda-\beta)||z||_2^2~~~\forall \lambda>\beta,
\end{equation*}
and thus $z=0$ which contradicts our assumption. This completes the proof of (b). Therefore, the operator $A(u)$ is quasi-m-accretive.
\end{proof}

\begin{lemma}
For every $\omega\in H_p^s$ with $s>3/2$, $A(u)$ is bounded linear operator from $H_p^s$ to $L^{2}[0,1]$ and
\begin{displaymath}
||(A(u)-A(v))\omega||_{L^2[0,1]}\leq c_1||u-v||_{L^2[0,1]}||\omega||_{H_p^{s}}.
\end{displaymath}
\end{lemma}

\begin{proof}
Given $\omega\in H_p^s$ with $s>3/2$,
\begin{align*}
||(u\partial_x+\partial_x)\omega||_{L^2[0,1]}& \leq ||u\partial_x\omega||_{L^2[0,1]}+||\partial_x\omega||_{L^2[0,1]}\\ &\leq ||u||_{L^2[0,1]}||\partial_x\omega||_{L^{\infty}[0,1]}+||\partial_x\omega||_{L^{\infty}[0,1]} \\
&\leq ||u||_{L^2[0,1]}||\partial_x\omega||_{H_p^{s-1}}+||\partial_x\omega||_{H_p^{s-1}}\leq c_1 ||u||_{L^2[0,1]}||\omega||_{H_p^{s}},
\end{align*}
in view of the inclusion property of Lebesgue spaces for finite intervals and Sobolev embedding theorems. Assumption (A2) follows from replacing $u$ by $u-v$ in the inequality.
\end{proof}
Before we verify (A3), we give a lemma needed for the proof:
\begin{lemma}
\label{Mlt} \cite{KatoII}
Let $f\in H_p^s$, $s>3/2$ and $M_f$ be the multiplication operator by $f$. Then, for $|\tilde{t}|,|\tilde{s}|\leq s-1$,
\begin{displaymath}
||\Lambda^{-\tilde{s}}[\Lambda^{\tilde{s}+\tilde{t}+1},M_f]\Lambda^{-\tilde{t}}\omega||_{L^2[0,1]}\leq c||f||_{H_p^s}||\omega||_{L^2[0,1]}.
\end{displaymath}
where $[,]$ represents the usual commutator of the linear operators. 
\end{lemma}
Now, we define a bounded linear operator:
\begin{lemma}
The operator
$$B(u)=\Lambda^{s}(u\partial_x+\partial_x)\Lambda^{-s}-(u\partial_x+\partial_x)=[\Lambda^{s},u\partial_x+\partial_x]\Lambda^{-s}$$
is bounded in $L^{2}[0,1]$ for $u\in H_p^s$ with $s>3/2$.
\end{lemma}

\begin{proof}
Note that
$$
\Lambda^{s}(u\partial_x+\partial_x)\Lambda^{-s}-(u\partial_x+\partial_x)
=\Lambda^{s}u\partial_x\Lambda^{-s}+\Lambda^{s}\partial_x\Lambda^{-s}-(u\partial_x+\partial_x)=[\Lambda^{s},u\partial_x]\Lambda^{-s}
$$
since $\partial_x$ and $\Lambda$ commute. Moreover, we have
$[\Lambda^{s},u\partial_x]\Lambda^{-s}=[\Lambda^{s},u]\Lambda^{-s}\partial_x$, so that
\begin{eqnarray}\label{map}
||B(u)\omega||_{L^2[0,1]} &= ||[\Lambda^{s},u]\Lambda^{-s}\partial_x\omega||_{L^2[0,1]}
= ||[\Lambda^{s},u]\Lambda^{1-s}\Lambda^{-1}\partial_x\omega||_{L^2[0,1]} \nonumber\\
&\leq ||u||_{H_p^s}||\Lambda^{-1}\partial_x\omega||_{L^2[0,1]}=||u||_{H_p^s}||\omega||_{L^2[0,1]}.
\end{eqnarray}
in view of Lemma \ref{Mlt} for $\tilde{s}=0$ and $\tilde {t}=s-1$:
\end{proof}

\begin{remark}
If we replace $u$ with $u-v$ in (\ref{map}), it can be easily observed that
\begin{displaymath}
||B(u)-B(v)\omega||_{L^2[0,1]}\leq ||\omega||_{L^2[0,1]}||u-v||_{H_p^s},
\end{displaymath}
which proves (A3).
\end{remark}

\begin{lemma}
Let $f(u)$ be given by (\ref{f}). Then:
\begin{itemize}
\item[(i)]$||f(u)||_{H_p^s}\leq M$ for some constant $M$ depending on $||u||_{H_p^s}$.
\item[(ii)]$||f(u)-f(v)||_{L^2[0,1]}\leq c_3||u-v||_{L^2[0,1]}$.
\item[(iii)]$||f(u)-f(v)||_{H_p^s}\leq c_4||u-v||_{H_p^s}$, $s>3/2$.
\end{itemize}
\end{lemma}

\begin{proof}
Observe that
\begin{displaymath}
f(u)= -(1-\frac{\mu}{12}\partial_x^2)^{-1}\partial_x (g(u))=-\partial_x (P*(g(u)))
\end{displaymath}
where * denotes the convolution and $P(x)$ stands for the Green's function of the operator \Big($1-\frac{\mu}{12}\partial_x^2$\Big) in the periodic case. Therefore,
\begin{align*}
||f(u)-f(v)||_{L^2[0,1]} =& ||-\partial_x (P*(g(u)-g(v)))||_{L^2[0,1]}\\
\leq& ||(u-v)||_{L^2[0,1]}+||(u-v)(u+v)||_{L^2[0,1]}\\
&+||(u-v)(u^2+uv+v^2)||_{L^2[0,1]}\\
&+||(u-v)(u+v)(u^2+v^2)||_{L^2[0,1]}\\
&+||\partial_x(u-v)\partial_x(u+v)||_{H_p^{-1}}
\end{align*}
Using the imbedding property of Sobolev spaces $H^s(\mathbb{R})$, i.e. $||.||_{s_1}\leq ||.||_{s_2}$ if $s_1\leq s_2$, Cauchy-Schwartz inequality; and Sobolev embedding theorem, we obtain
\begin{align*}
||f(u)-f(v)||_{L^2[0,1]} 
            \leq& ||(u-v)||_{L^2[0,1]}+||(u+v)||_{s}||(u-v)||_{L^2[0,1]}\\
            &+||(u^2+uv+v^2)||_{H_p^{s}}||(u-v)||_{L^2[0,1]}\\
            &+||(u+v)(u^2+v^2)||_{H_p^{s}}||(u-v)||_{L^2[0,1]}\\
            &+||(u+v)||_{H_p^{s}}||(u-v)||_{L^2[0,1]}\\
            \leq& c_3 ||u-v||_{L^2[0,1]}
\end{align*}
where $c_3$ is a constant depending on $||u||_{H_p^{s}}$ and $||v||_{H_p^{s}}$. This proves (ii). To prove (iii), we have the following estimates:
\begin{align*}
||f(u)-f(v)||_{H_p^s} =& ||(1-\frac{\mu}{12}\partial_x^2)^{-1}\partial_x(g(u)-g(v))||_{H_p^s}\\
                  \leq& ||(u-v)||_{H_p^{s-1}}+||(u-v)(u+v)||_{H_p^{s-1}}\\
                  &+||(u-v)(u^2+uv+v^2)||_{H_p^{s-1}}\\
                  &+||(u-v)(u+v)(u^2+v^2)||_{H_p^{s-1}}\\
                  &+||\partial_x(u-v)\partial_x(u+v)||_{H_p^{s-1}}\\
                  \leq& ||(u-v)||_{H_p^s}+||(u-v)||_{H_p^s}||(u+v)||_{H_p^s}\\
                  &+||(u-v)||_{H_p^s}||(u^2+uv+v^2)||_{H_p^s}\\
                  &+||(u-v)||_{H_p^s}||(u+v)(u^2+v^2)||_{H_p^s}\\
                  &+||\partial_x(u-v)||_{H_p^{s-1}}||\partial_x(u+v)||_{H_p^{s-1}}\\
                  \leq& c_4 ||u-v||_{H_p^s}
\end{align*}
where $c_4$ is also a constant depending on $||u||_{H_p^s}$ and $||v||_{H_p^s}$. Since we choose $u_0\in H_p^s$, this estimate actually proves continuous dependence on the initial data. Note that (i) can be obtained from (iii) by choosing $v=0$. Hence, the estimates for (A4)are satisfied.
\end{proof}
\noindent
 \title{\textbf{Proof of Proposition \ref{Lwp}}}
Since we have $u_0\in H_p^s$, $s>3/2$, and the assumptions (A1)-(A4) hold for $X=L^2[0,1]$ and $Y=H_p^s$, $s>3/2$, we get $ u\in C([0,T), H_p^s)\cap C^1([0,T),L^{2}[0,1])$ in view of Theorem \ref{kato}.

\qed

\begin{remark}
\label{regularity}
In view of equation (\ref{SG}), we also have more regular solutions $ u\in C([0,T), H_p^s)\cap C^1([0,T),H_p^{s-1}])$, $s>3/2$.
\end{remark}
 \section{Wave Breaking}
 
 \subsection{Existence of breaking waves}
 
In the following, we deduce that for solutions of the evolution equation for surface waves, 
  \begin{align}
    \label{MAE}
  & u_t + u_x + \frac{3}{2}\varepsilon u u_x -\frac{3}{8}\varepsilon^{2} u^2u_x + \frac{3}{16}\varepsilon^{3} u^3u_x+\frac{\mu}{12}(u_{xxx}-u_{xxt}) \nonumber\\
  & +\frac{7\varepsilon}{24}\mu (uu_{xxx} +2 u_xu_{xx})=0,
  \end{align} 
singularities can occur in finite time only in the form of wave breaking, more specifically surging breakers. In other words, there exists a breaking time for the solution where the slope of the wave becomes infinite despite the fact that the wave remains bounded.

\begin{proposition}
\label{brk}
 If, for some initial data $u_0\in H_p^2$, the maximal existence time $T>0$ of the periodic solution to (\ref{MAE}) is finite, then the solution $u(x,t)\in C([0,T), H_p^2)\cap C^1([0,T),L^{2}[0,1])$ has the property that
 \begin{equation}
  \label{sup}
  \sup_{t\in[0,T),~x\in[0,1]}\{|u(x,t)|\}<\infty
 \end{equation}
whereas 
\begin{equation}
\label{slope}
 \limsup_{t\uparrow T}\{u_x(x,t)\}=+\infty.
\end{equation}

\end{proposition}
\begin{proof}
Multiplying equation (\ref{MAE}) by $u$ and integrating over $[0,1]$, we find that 
\begin{equation}
  \label{E}
    E(u)= \frac{1}{2}\int_{0}^{1} \Big(u^2 +\frac{\mu}{12} u_x^2\Big) dx
  \end{equation}
is a conserved quantity of (\ref{MAE}). Hence, the boundedness (\ref{sup}) of the solution $u$ is ensured by $E(u)$ and the imbedding $H_p^1\subset L^{\infty}[0,1]$. In view of Proposition \ref{Lwp}, and using the fact that $H_p^2$ is dense in $H_p^s$ for $3/2<s<2$, a periodic solution to (\ref{MAE}) has a finite maximal existence time if and only if $||u||_{H_p^2}$ blows up in finite time. Therefore, to prove (\ref{slope}) and conclude that the maximal existence time is finite, it is sufficient to show that if we can find a constant $M_1=M_1(u_0)$ such that 
\begin{equation}
\label{bdd}
 u_x(x,t)\leq M_1,~~~x\in[0,1]
\end{equation}
as long as the solution is defined, $||u||_{H_p^2}$ will remain bounded for finite time. For that purpose, we multiply (\ref{MAE}) by $u_{xx}$ and integrate over $[0,1]$ to get
\begin{align*}
 \int_0^1 (u_t u_{xx}+ u_xu_{xx} + \frac{3}{2}\varepsilon u u_xu_{xx} -\frac{3}{8}\varepsilon^{2} u^2u_xu_{xx} + \frac{3}{16}\varepsilon^{3} u^3u_xu_{xx}\\
 +\frac{\mu}{12}(u_{xx}u_{xxx}-u_{xx}u_{xxt})+\frac{7\varepsilon}{24}\mu (uu_{xx}u_{xxx} +2 u_xu^2_{xx}))dx=0.
\end{align*}
Taking advantage of the periodicity, 
\begin{align*}
 \int_0^1 (u_xu_{xx}+\frac{\mu}{12}u_{xx}u_{xxx})dx=\frac{1}{2}\int_0^1 ((u_x^2)_x+\frac{\mu}{12}(u_{xx}^2)_x)dx=0.
\end{align*}
Moreover, integration by parts gives
\begin{align*}
 \frac{1}{2}\partial_t\Big(\int_0^1 (u_{x}^{2}+\frac{\mu}{12}u_{xx}^2)dx\Big)
 =\int_0^1 &\Big(\frac{3}{2}\varepsilon u u_xu_{xx} -\frac{3}{8}\varepsilon^{2} u^2u_xu_{xx} + \frac{3}{16}\varepsilon^{3} u^3u_xu_{xx}\\
 &+\frac{7\varepsilon}{24}\mu (uu_{xx}u_{xxx} +2 u_xu_{xx}^2)\Big)dx.
\end{align*}
Now, choose $M_0$ such that 
\begin{equation*}
 |u(x,t)|\leq M_0,~~~x\in[0,1],
\end{equation*}
and assume that (\ref{bdd}) holds. By Cauchy-Schwarz and Young's inequalities,
\begin{align*}
 \frac{1}{2}\partial_t\Big(\int_0^1 (u_{x}^{2}+\frac{\mu}{12}u_{xx}^2)dx\Big)
&\leq\Big(\frac{3}{4}\varepsilon M_0+\frac{3}{16}\varepsilon^{2}M_0^2+\frac{3}{32}\varepsilon^{3}M_0^3\Big)\int_0^1 u_x^2 dx\\
&+\Big(\frac{3}{4}\varepsilon M_0+\frac{3}{16}\varepsilon^{2}M_0^2+\frac{3}{32}\varepsilon^{3}M_0^3+\frac{21}{48}\varepsilon\mu M_1\Big)\int_0^1 u_{xx}^2 dx.
\end{align*}
Define 
\begin{equation}
\label{H}
H(t):=\frac{1}{2}\int_0^1\Big(u^2+\frac{\mu}{6}u_{x}^{2}+\frac{\mu^2}{144}u_{xx}^2 \Big)dx,
\end{equation}
and observe that
\begin{align*}
\partial_t H(t)=&\frac{1}{2}\partial_t\Big(\int_0^1 (u^{2}+\frac{\mu}{12}u_{x}^2)dx\Big)+ \frac{\mu}{24}\partial_t\Big(\int_0^1 (u_{x}^{2}+\frac{\mu}{12}u_{xx}^2)dx\Big)\\
 \leq&\frac{\mu}{12}\Big(\frac{6}{\mu}(\frac{3}{4}\varepsilon M_0+\frac{3}{16}\varepsilon^{2}M_0^2+\frac{3}{32}\varepsilon^{3}M_0^3)\\
 &+\frac{144}{\mu^2}(\frac{3}{4}\varepsilon M_0+\frac{3}{16}\varepsilon^{2}M_0^2+\frac{3}{32}\varepsilon^{3}M_0^3+\frac{21}{48}\varepsilon\mu M_1)\Big)H(t).
\end{align*}
In view of Gronwall's inequality, $||u||_{H_p^2}$ does not blow up in finite time under the assumption (\ref{bdd}).
\end{proof}
\subsection{Blow-up Result}

Notice that (\ref{MAE}) is a quasi-linear evolution equation which may be rewritten in the form 
   \begin{align}
    \label{QL}
  & u_t - u_x - \frac{7}{2}\varepsilon u u_x \nonumber\\
  &+\Big(1-\frac{\mu}{12}\partial_x^2\Big)^{-1}\partial_x\Big(2u+\frac{5}{2}\varepsilon u^2-\frac{1}{8}\varepsilon^{2} u^3+ \frac{3}{64}\varepsilon^{3} u^4-\frac{7}{48}\varepsilon\mu u_x^2\Big)\nonumber\\
  &=u_t - u_x - \frac{7}{2}\varepsilon u u_x+\partial_x(P*g(u))=0,
  \end{align}
where $P(x)$ is the Green function of the operator $\Big(1-\frac{\mu}{12}\partial_x^2\Big)$ in the periodic case. It is given by 
\begin{equation}
\label{kernel}
 P(x)=\sqrt{\frac{3}{\mu}}\frac{e^{2\sqrt{\frac{3}{\mu}}(x-[x])}+e^{2\sqrt{\frac{3}{\mu}}(1-(x-[x])}}{e^{2\sqrt{\frac{3}{\mu}}}-1}, ~~~x\in\mathbb{R},
\end{equation}
with

\begin{displaymath}
\Vert P(x)\Vert_{L^{1}[0,1]}=\frac{\mu}{3}=:n_1,
\end{displaymath}

\begin{displaymath}
\Vert P(x)\Vert_{L^{2}[0,1]}=(\frac{3}{4\mu})^{1/4}\Big(\frac{e^{4\sqrt{\frac{3}{\mu}}}+4\sqrt{\frac{3}{\mu}}e^{2\sqrt{\frac{3}{\mu}}}-1}{e^{4\sqrt{\frac{3}{\mu}}}-2e^{2\sqrt{\frac{3}{\mu}}}+1}\Big)^{1/2}=:n_2,
\end{displaymath}
and
\begin{displaymath}
\Vert P(x)\Vert_{L^{\infty}[0,1]}=\sqrt{\frac{3}{\mu}}\Big(\frac{e^{2\sqrt{\frac{3}{\mu}}}+1}{e^{2\sqrt{\frac{3}{\mu}}}-1}\Big)=:n_{\infty}
\end{displaymath}
where $n_1,~n_2,~n_{\infty}$ are real constants which correspond to the finite norms of $P(x)$. The evaluation of the kernel (\ref{kernel}) relies on ordinary differential equation theory. Consider
\begin{equation}
\label{ODE}
(1-\frac{\mu}{12}\partial_x^2)u(x)=f(x).
\end{equation}
In general, the solution of (\ref{ODE}) is given as
\begin{align*}
    u(x)&=u_{homogenous}+u_{particular}\\
	&=c_1e^{2\sqrt{\frac{3}{\mu}}x}+c_2e^{-2\sqrt{\frac{3}{\mu}}x}+\frac{3}{\mu}\int_\mathbb{R}e^{-2\sqrt{\frac{3}{\mu}}|x-y|}f(y)dy.
\end{align*} 
where $c_1$ and $c_2$ are real constants. Note that $\{\cosh(2\sqrt{\frac{3}{\mu}}x),\sinh(2\sqrt{\frac{3}{\mu}}x)\}$ is equivalent to the solution set used for the homogenous part of the solution. Moreover, without loss of generality, $f(x)$ can be assumed to have value zero outside the interval $[0,1]$, since we work on a periodic problem. Hence, the solution can be given as

\begin{eqnarray}
\label{solution}
u(x)=c_1\cosh\Big(2\sqrt{\frac{3}{\mu}}x\Big)+c_2\sinh\Big(2\sqrt{\frac{3}{\mu}}x\Big)+\frac{3}{\mu}\int_0^1 e^{-2\sqrt{\frac{3}{\mu}}|x-y|}f(y)dy.
\end{eqnarray} 

\noindent
To determine the unknown constants $c_1$ and $c_2$, we first take into account periodicity and evaluate $u(0)$ and $u(1)$, since they are the equal boundary values for equation (\ref{ODE}) in $[0,1]$. Following the same procedure for $u_x$ and equating $u_x(0)$ and $u_x(1)$, we get two equalities involving the constants $c_1$, $c_2$ and $f(x)$. Then, we solve these two equalities for $c_1$ and $c_2$, and find their expressions depending on $f$. Plugging the resulting values in (\ref{solution}), and collecting the coefficients of $f$ finally gives the kernel $P(x)$.

Now, we will ensure that some wave solutions to (\ref{QL}) blow up in finite time in the form of surging breakers given as in Proposition \ref{brk}. Our key theorem which guarantees the existence of at least one real valued point where the supremum of the slope is attained is given below. The idea is to show that this supremum is increasing up to maximal existence time and afterwards, roughly speaking, the wave breaks and the supremum goes to infinity. 

\begin{theorem} \cite{CE98} \label{supremum}
 Let $T>0$ and $u\in C^1([0,T);H_p^2)$. Then, for every $t\in[0.T)$, there exists at least one point $\xi(t)\in\mathbb{R}$ with 
 \begin{equation}
 \label{S}
    S(t):=\sup_{x\in\mathbb{R}}[u_x(x,t)]=u_x(\xi(t),t),
 \end{equation}
and the function $S$ is almost everywhere differentiable on $(0,T)$ with 
\begin{displaymath}
 \frac{dS}{dt}=u_{tx}(\xi(t),t)~~~~a.e.~~ on~~ (0,T).
\end{displaymath}

\end{theorem}

It follows from Remark \ref{regularity} that we have enough degree of smoothness for the solution $u(x,t)$ to apply Theorem \ref{supremum}. As a preliminary observation, we differentiate (\ref{QL}) with respect to the spatial variable $x$ and obtain:

\begin{equation*}
u_{tx}-u_{xx}-\frac{7}{2}\varepsilon u_x^2-\frac{7}{2}\varepsilon uu_{xx}+\partial_x^2(P*g(u))=0.
\end{equation*}
Moreover, we obtain 
\begin{displaymath}
\partial_x^2(P*g)= \frac{12}{\mu}P*g-\frac{12}{\mu}g.
\end{displaymath}
Therefore, we get
\begin{equation}
\label{diff}
u_{tx}-u_{xx}-\frac{7}{2}\varepsilon u_x^2-\frac{7}{2}\varepsilon uu_{xx}+\frac{12}{\mu}(P*g(u))-\frac{12}{\mu}g(u)=0.
\end{equation}

The blow-up result for equation (\ref{QL}) is then proved as follows:
\begin{proposition}
 If the initial wave profile $u_0\in H_p^3$ satisfies
 \begin{align}
 \label{iniprof}
 \lvert\inf_{x\in[0,1]}\{\partial_x u_0(x)\}\rvert^2 >&
 \frac{12}{\mu\varepsilon}\Big(2n_2C_0^{1/2}+\frac{5}{2}\varepsilon n_{\infty}C_0+\frac{1}{8}\varepsilon^2 n_{\infty}\sqrt{\frac{13}{\mu}}C_0^{3/2}\nonumber\\
 &+\frac{3}{64}\varepsilon^3n_{\infty}\frac{13}{\mu}C_0^2+\frac{7}{4}\varepsilon n_{\infty}C_0+2\sqrt{\frac{13}{\mu}}C_0^{1/2}\nonumber\\
 &+\frac{5}{2}\varepsilon\frac{13}{\mu}C_0+\frac{1}{8}\varepsilon^2(\frac{13}{\mu})^{3/2}C_0^{3/2}+\frac{3}{64}\varepsilon^3(\frac{13}{\mu})^2C_0^{2}\Big)
 \end{align}
where 
\begin{displaymath}
 C_0=\int_0^1 (u_0^2+\frac{\mu}{12}u_{0x}^2)dx>0,
\end{displaymath}
then wave breaking occurs for the solutions of (\ref{QL}) in finite time $T=O(\frac{1}{\varepsilon})$.
\end{proposition}
\begin{proof}
Since (\ref{diff}) is an equality in the space of continuous functions derived for any instant $t$, we can also evaluate the equality at a point $\xi(t)\in \mathbb{R}$ at some fixed time $t$ using (\ref{S}). Besides, $u_{xx}(\xi(t),t)=0$ due to the fact that $H_p^3\subset C^2([0,1])$, and hence $u$ is $C^2$ in the spatial variable, and $\xi(t)$ gives an extremum for $u_x$. 
Thus, equation (\ref{diff}) at $(\xi(t),t)$ becomes 
\begin{equation}
\label{instant}
 S'(t)-\frac{7}{2}\varepsilon S^2(t)= -\frac{12}{\mu}(P*g(u))+\frac{12}{\mu}g(u)
\end{equation}
where 
\begin{displaymath}
 g(u)=2u+\frac{5}{2}\varepsilon u^2-\frac{1}{8}\varepsilon^{2} u^3+ \frac{3}{64}\varepsilon^{3} u^4-\frac{7}{48}\varepsilon\mu u_x^2.
\end{displaymath}
In view of Lemma 2 in \cite{Con00}, for $u\in H_p^3$, 
\begin{equation}
 \max_{x\in[0,1]}u^2(x)\leq \frac{13}{\mu}C_0.
\end{equation}
Furthermore, by Young's inequality, we have
\begin{align*}
&\Vert P*u\Vert_{L^{\infty}[0,1]}\leq \Vert P\Vert_{L^{2}[0,1]}\Vert u\Vert_{L^{2}[0,1]}\leq n_2 C_0^{1/2},\\
&\Vert P*u^2\Vert_{L^{\infty}[0,1]}\leq \Vert P\Vert_{L^{\infty}[0,1]}\Vert u^2\Vert_{L^{1}[0,1]}\leq \Vert P\Vert_{L^{\infty}[0,1]}\Vert u\Vert_{L^{2}[0,1]}^2\leq n_{\infty}C_0,\\
&\Vert P*u^3\Vert_{L^{\infty}[0,1]}\leq \Vert P\Vert_{L^{\infty}[0,1]}\Vert u\Vert_{L^{\infty}[0,1]}\Vert u\Vert_{L^{2}[0,1]}^2\leq n_{\infty}\sqrt{\frac{13}{\mu}}C_0^{1/2}C_0,\\
&\Vert P*u^4\Vert_{L^{\infty}[0,1]}\leq\Vert P\Vert_{L^{\infty}[0,1]}\Vert u^2\Vert_{L^{\infty}[0,1]}\Vert u\Vert_{L^{2}[0,1]}^2\leq n_{\infty}\frac{13}{\mu}C_0C_0,\\
&\Vert P*u_x^2\Vert_{L^{\infty}[0,1]}\leq\Vert P\Vert_{L^{\infty}[0,1]}\Vert u_x\Vert_{L^{2}[0,1]}^2\leq n_{\infty}\frac{12}{\mu}C_0,
\end{align*}
which yields two differential inequalities for the locally Lipschitz function $S(t)$:
\begin{align}
\label{leq}
&S'(t)-\frac{7}{4}\varepsilon S^2(t)\leq\nonumber\\
& \frac{12}{\mu}\Big(2n_2C_0^{1/2}+\frac{5}{2}\varepsilon n_{\infty}C_0+\frac{1}{8}\varepsilon^2 n_{\infty}\sqrt{\frac{13}{\mu}}C_0^{3/2}+\frac{3}{64}\varepsilon^3n_{\infty}\frac{13}{\mu}C_0^2+\frac{7}{4}\varepsilon n_{\infty}C_0\Big)\nonumber\\
&+\frac{12}{\mu}\Big(2\sqrt{\frac{13}{\mu}}C_0^{1/2}+\frac{5}{2}\varepsilon\frac{13}{\mu}C_0+\frac{1}{8}\varepsilon^2(\frac{13}{\mu})^{3/2}C_0^{3/2}+\frac{3}{64}\varepsilon^3(\frac{13}{\mu})^2C_0^{2}\Big)
\end{align}
and
\begin{align}
\label{geq}
&S'(t)-\frac{7}{4}\varepsilon S^2(t)\geq\nonumber\\
& -\frac{12}{\mu}\Big(2n_2C_0^{1/2}+\frac{5}{2}\varepsilon n_{\infty}C_0+\frac{1}{8}\varepsilon^2 n_{\infty}\sqrt{\frac{13}{\mu}}C_0^{3/2}+\frac{3}{64}\varepsilon^3n_{\infty}\frac{13}{\mu}C_0^2+\frac{7}{4}\varepsilon n_{\infty}C_0\Big)\nonumber\\
&-\Big(\frac{24}{\mu}\sqrt{\frac{13}{\mu}}C_0^{1/2}+\frac{3}{2\mu}\varepsilon^2(\frac{13}{\mu})^{3/2}C_0^{3/2}\Big)
\end{align}
for almost every $t\in(0,T)$. Notice that $u_0\neq 0$ ensures that $S(0)>0$. By our assumption (\ref{iniprof}), at $t=0$, the right hand side of the inequality (\ref{geq}) is larger than $-\varepsilon S^2(0)$. Thus, we infer that up to the maximal existence time $T>0$ of the solution, the function $S(t)$ must be strictly increasing, and moreover
\begin{displaymath}
 S'(t)\geq \frac{3}{4}\varepsilon S^2(t)~~for~ a.e.~t\in (0,T).
\end{displaymath}
Dividing by 
\begin{displaymath}
 S^2(t)\geq S^2(0)>0,~~t\in[0,T),
\end{displaymath}
and integrating, we get 
\begin{displaymath}
 \frac{1}{S(t)}\leq \frac{1}{S(0)}-\frac{3}{4}\varepsilon t, ~~t\in (0,T).
\end{displaymath}
Since $S(t)>0$, we must have 
\begin{displaymath}
 \lim_{t\nearrow T} S(t)=\infty,
\end{displaymath}
which reveals that singularities will occur in the form of surging breakers, and
\begin{equation}
 \label{order1}
 T\leq \frac{4}{3\varepsilon S(0)}.
\end{equation}
Additionally, the inequality (\ref{leq}) yields
\begin{displaymath}
 S'(t)\leq \frac{11}{4}\varepsilon S^2(t)~~for~ a.e.~t\in (0,T)
\end{displaymath}
under the same assumption on the initial profile. Hence,
\begin{displaymath}
 \frac{1}{S(0)}-\frac{1}{S(t)}\leq\frac{11}{4}\varepsilon t, ~~t\in (0,T).
\end{displaymath}
Since $\lim_{t\nearrow T} S(t)=\infty$, we get
\begin{equation}
 \label{order2}
 T\geq \frac{4}{11\varepsilon S(0)}.
\end{equation}
From the time estimates (\ref{order1}) and (\ref{order2}), we deduce that the finite maximal existence time $T>0$ is of order $O(\frac{1}{\varepsilon})$.

\end{proof}

\section{Acknowledgement}
The support of The Scientific and Technological Research Council of Turkey (TUBITAK) under ``International Postdoctoral Research Scholarship Programme'' is gratefully acknowledged.

 \end{document}